\documentclass{article}
\usepackage{arxiv}

\usepackage[utf8]{inputenc} %
\usepackage[T1]{fontenc}    %
\usepackage{hyperref}       %
\usepackage{url}            %
\usepackage{booktabs}       %
\usepackage{amsfonts}       %
\usepackage{nicefrac}       %
\usepackage{microtype}      %
\usepackage{lipsum}
\usepackage{mathtools}
\usepackage{cite}
\usepackage{amsmath}
\usepackage{amssymb}
\usepackage{amsthm}
\usepackage{algorithm,algcompatible}
\usepackage[toc,page]{appendix}
\newcommand{\iu}{\mathrm{i}\mkern1mu}

\usepackage{lineno}

\newtheorem{theorem}{Theorem}
\newtheorem{lemma}{Lemma}
\newtheorem{definition}{Definition}
\newtheorem{corollary}{Corollary}

\usepackage{mathptmx}      %
\begin{document}

\title{High-Order Approximation Rates for Shallow \\ Neural Networks with Cosine and ReLU$^k$ Activation Functions
}

\author{Jonathan W. Siegel \\
  Department of Mathematics\\
  Pennsylvania State University\\
  University Park, PA 16802 \\
  \texttt{jus1949@psu.edu} \\
  \And Jinchao Xu \\
  Department of Mathematics\\
  Pennsylvania State University\\
  University Park, PA 16802 \\
  \texttt{jxx1@psu.edu} \\
}

\maketitle

\begin{abstract}
 We study the approximation properties of shallow neural networks with an activation function which is a power of the rectified linear unit. Specifically, we consider the dependence of the approximation rate on the dimension and the smoothness in the spectral Barron space of the underlying function $f$ to be approximated. We show that as the smoothness index $s$ of $f$ increases, shallow neural networks with ReLU$^k$ activation function obtain an improved approximation rate up to a best possible rate of $O(n^{-(k+1)}\log(n))$ in $L^2$, independent of the dimension $d$. The significance of this result is that the activation function ReLU$^k$ is fixed independent of the dimension, while for classical methods the degree of polynomial approximation or the smoothness of the wavelets used would have to increase in order to take advantage of the dimension dependent smoothness of $f$. In addition, we derive improved approximation rates for shallow neural networks with cosine activation function on the spectral Barron space. Finally, we prove lower bounds showing that the approximation rates attained are optimal under the given assumptions.
\end{abstract}


\section{Introduction}
We consider approximating a function $f:\Omega\rightarrow \mathbb{C}$, where $\Omega\subset \mathbb{R}^d$ is a bounded domain, by a superposition of ridge functions of the form
\begin{equation}\label{ridge-superposition}
 f_n(x) = \sum_{i=1}^n a_i\sigma(\omega_i\cdot x + b_i),
\end{equation}
with activation function $\sigma = \cos(x)$ or $\sigma = [\max(0,x)]^k$ for $k\geq 0$. The latter case corresponds to a neural network with a single hidden layer and activation function given by a power of a rectified linear unit \cite{nair2010rectified}, or rectified power unit \cite{li2019better}. In the case where $k=0$, $[\max(0,x)]^0$ is interpreted as the Heaviside function and we may also use any sigmoidal activation function in its place \cite{barron1993universal}. Approximation by functions of the form \eqref{ridge-superposition} has recieved a significant amount of attention in the literature. For instance, it has been shown that as long as $\sigma$ is not a polynomial, functions of the form \eqref{ridge-superposition} are dense in $C(\Omega)$ \cite{attali1997approximations,leshno1993multilayer,ellacott1994aspects} and $C^k(\Omega)$ \cite{hornik1991approximation,hornik1993some}, and in \cite{costarelli2015approximation,costarelli2013approximation} explicit operators realizing the approximation for any absolutely continuous $f$ are constructed.

Beyond the problem of density, we are interested in determining rates of approximation for a given class of functions $f:\Omega\rightarrow \mathbb{C}$ by an expression of the form \eqref{ridge-superposition} with respect to the Sobolev norm $H^m(\Omega)$. Typical results of this type consider functions $f$ which either satisfy classical smoothness assumptions, such as membership in a suitable Sobolev space, or non-standard smoothness assumptions coming from the theory of non-linear approximation by a dictionary.

For example, results for functions $f$ in high-order Sobolev spaces have been obtained in \cite{petrushev1998approximation}. Here it is shown that if the activation function $\sigma$ achieves approximation order $O(n^{-r})$ for one-dimensional functions $f\in H^r([0,1])$, then an approximation rate of $O(n^{-\frac{1}{2}-\frac{2r-1}{2d}})$ can be attained for high dimensional functions $f\in H^{\frac{d-1}{2} + r}(B^d)$ on the unit ball with respect to $L^2$.  In \cite{maiorov1999best} the case where each term in \eqref{ridge-superposition} may have a different profile $\sigma$ is considered and the optimal rates in this case are derived for all Sobolev spaces $H^r$ in $\mathbb{R}^d$. This result is generalized in \cite{gordon2001best,maiorov2010best} to general $L^p$ spaces.

A typical example of a non-standard smoothness assumption is membership in the closed convex hull of a suitable bounded dictionary $\mathbb{D}\subset H^m(\Omega)$, specifically one considers
\begin{equation}
 f\in B_1(\mathbb{D}) = \overline{\left\{\sum_{i=1}^n a_id_i,~d_i\in \mathbb{D},~n\in \mathbb{N},~\sum_{i=1}^n |a_i| = 1\right\}}.
\end{equation}
One can also characterize this set using the guage norm (see, for instance \cite{rockafellar1970convex}) of $B_1(\mathbb{D})$, which we denote by $\mathcal{K}_1(\mathbb{D})$ (following the notation from \cite{devore1998nonlinear})
\begin{equation}
\|f\|_{\mathcal{K}_1(\mathbb{D})} = \inf\{c > 0:~f\in cB_1(\mathbb{D})\}.
\end{equation}
Here $B_1(\mathbb{D})$ is exactly the unit ball in the norm $\mathcal{K}_1(\mathbb{D})$.
For this class of functions, one considers non-linear approximation by finite dictionary expansions, i.e. approximation form the set
\begin{equation}
 \Sigma_n(\mathbb{D}) = \left\{\sum_{i=1}^n a_id_i,~d_i\in \mathbb{D}\right\}.
\end{equation}
When the dictionary $\mathbb{D}$ is of the form
\begin{equation}
\mathbb{D}_\sigma = \{\sigma(\omega\cdot x + b),~\omega\in \mathbb{R}^d,~b\in\mathbb{R}\},
\end{equation}
this exactly corresponds to an expansion of the form \eqref{ridge-superposition}. For some activation functions, such as $[\max(0,x)]^k$ for $k > 0$, the dictionary $\mathbb{D}_\sigma$ is not bounded. In this case, the dictionary must be modified (details can be found in \cite{siegel2021characterization,siegel2021sharp}) and the resulting space, called the Barron space \cite{ma2019barron} or variation space corresponding to shallow ReLU networks \cite{bach2017breaking} when $k=1$, is closely related to the ridgelet spaces introduced in \cite{candes1998ridgelets} (to be precise, it is sandwiched between $R_{1,1}^{1+k+(d-1)/2}$ and $R_{1,\infty}^{1+k+(d-1)/2}$, see section 4.2 in \cite{candes1998ridgelets}).

Using a classical probabilistic argument of Maurey \cite{pisier1981remarques}, an approximation rate of $O(n^{-\frac{1}{2}})$ can be obtained for the class $B_1(\mathbb{D})$ using non-linear dictionary expansions. Moreover, Jones \cite{jones1992simple} gave a constructive proof of this fact using the relaxed greedy algorithm and applied this result to shallow neural networks with a cosine activation function. Improvements upon this rate of dictionary approximation under an assumption about the behavior of the relaxed greedy algorithm appear in \cite{kurkova2008geometric,lavretsky2002geometric}. These results yield exponential rates of convergence for individual functions in the convex hull of $\mathbb{D}$ (but not necessarily its closure), which are however not uniform over the class $B_1(\mathbb{D})$. Further, under compactness \cite{makovoz1996random,klusowski2018approximation} or smoothness \cite{siegel2021sharp} assumptions on the dictionary $\mathbb{D}$ improved rates can also be obtained, although for general dictionaries the Maurey-Jones rate is the best one can expect \cite{kurkova2001bounds}.

The application of the Jones-Maurey result to neural networks with sigmoidal activation function is due to Barron \cite{barron1993universal}. In this work, the relevant class of functions is
\begin{equation}\label{spectral-barron}
 \mathcal{B}^s(\Omega) = \left\{f:\Omega\rightarrow \mathbb{C}:~\|f\|_{\mathcal{B}^s(\Omega)} := \inf_{f_e|\Omega = f} \int_{\mathbb{R}^d}(1 + |\xi|)^s|\hat{f}_e(\xi)|d\xi < \infty\right\},
\end{equation}
where the infimum above is over extensions $f_e\in L^1(\mathbb{R}^d)$. Barron introduced this class for $s = 1$ and showed that for a sigmoidal activation function $\sigma$ we have $\mathcal{K}_1(\mathbb{D}_\sigma)\supset \mathcal{B}^1(\Omega)$. This shows that shallow neural networks with sigmoidal activation function can approximate functions satisfying a certain Fourier integrability condition with a rate of $O(n^{-\frac{1}{2}})$. This convergence rate for $f\in \mathcal{B}^1(\Omega)$ has been extended to a very general class of activation functions in \cite{hornik1994degree,siegel2020approximation}. A variety of other results for functions in the spectral Barron space \eqref{spectral-barron} have been obtained in \cite{klusowski2018approximation,bresler2020sharp,ma2019barron,CiCP-28-1707,makovoz1996random}, for instance.

It has been shown that the space $\mathcal{B}^s(\Omega)$ is exactly equivalent (with identical norm) to $B_1(\mathbb{F}_s^d)$ \cite{siegel2021characterization} for the dictionary
\begin{equation}
 \mathbb{F}_s^d = \{(1+|\omega|)^{-s}e^{2\pi \iu \omega\cdot x}:~\omega\in \mathbb{R}^d\}
\end{equation}
of decaying Fourier modes. There is a slight sublety here. Let $\Omega = [0,1]^d$, which is the case we are primarily interested in. Note that the frequency $\omega\in \mathbb{R}^d$ is allowed to be arbitrary. This results in a larger space than restricting $\omega\in \mathbb{Z}^d$, i.e. considering the dictionary
\begin{equation}
 \overline{\mathbb{F}_s^d} = \{(1+|\omega|)^{-s}e^{2\pi \iu \omega\cdot x}:~\omega\in \mathbb{Z}^d\},
\end{equation}
which is done for instance in \cite{lu2021priori}. Restricting $\omega$ to the integer lattice $\mathbb{Z}^d$ results in a much simpler space, since we have (see \cite{candes1998ridgelets}, section 7.2)
\begin{equation}
 \|f\|_{\mathcal{K}_1(\overline{\mathbb{F}_s^d})} \eqsim \sum_{n\in \mathbb{Z}^d} (1+|n|)^s||\hat{f}(n)|.
\end{equation}
However, by considering a pure frequency with $\omega\notin \mathbb{Z}^d$ and calculating its Fourier series, we can easily see that the above characterization does not hold for $\mathcal{B}^s(\Omega)$ and that the space $\mathcal{B}^s(\Omega)$ is strictly larger.

An interesting fact, first observed by Makovoz \cite{makovoz1996random}, is that for certain activation functions $\sigma$, the rate of approximation $O(n^{-\frac{1}{2}})$ derived by Barron \cite{barron1993universal} can be improved. In particular, Makovoz shows that for the Heaviside activation function 
\begin{equation}
\sigma = [\max(0,x)]^0 := \begin{cases} 
      0 & x\leq 0 \\
      1 & x > 0
   \end{cases}
\end{equation} 
the rate can be improved to $O(n^{-\frac{1}{2}-\frac{1}{2d}})$. 
Furthermore, when $\sigma = \max(0,x)^k$ for $k \geq 1$, this can be improved to
a rate of $O(n^{-\frac{1}{2}-\frac{2k+1}{2d}})$ \cite{klusowski2018approximation,CiCP-28-1707,siegel2021sharp} for the class $\mathcal{B}^{k+1}(\Omega)$. For functions $f\in \mathcal{B}^{1}(\Omega)$, improved rates have also been obtained for more general activation functions \cite{siegel2020approximation}. It has also been shown that the orthogonal greedy algorithm \cite{pati1993orthogonal} can constructively obtain such improved approximation rates \cite{siegel2021improved}.

In this work, we study how much further the rate of $O(n^{-\frac{1}{2}})$ can be improved given stronger assumptions on the smoothness of $f$, i.e. assuming that $f\in\mathcal{B}^s(\Omega)$ for larger values of $s$. 
By showing that the continuous Fourier transform of a function $f$ on the bounded set $\Omega$ can be replaced by a suitably chosen Fourier series, we show in Theorems \ref{approximation-rate-theorem} and \ref{spectral-convergence-theorem} that the approximation rates in $L^2$ with $\sigma = \cos(x)$ can be improved to $O(n^{-\frac{1}{2}-\frac{s}{d}})$ for $f\in\mathcal{B}^s(\Omega)$ with increasing $s$, and even that exponential convergence can be attained if the Fourier transform of $f$ decays rapidly enough.

To compare with the results in \cite{petrushev1998approximation}, we observe that $\sigma = \cos(x)$ attains approximation order $O(n^{-r})$ in $H^r([0,1])$ for any $r > 0$. This means that a rate of approximation of $O(n^{-\frac{1}{2}-\frac{s}{d}})$ for cosine networks already appear in \cite{petrushev1998approximation} for the Sobolev spaces $H^{\frac{d}{2} + s}(\Omega)$. However, our results apply to the spectral Barron space, which is not quite comparable, although we have $H^{\frac{d}{2} + s + \epsilon}(\Omega)\subset \mathcal{B}^s(\Omega)$ (see \cite{CiCP-28-1707} Lemma 2.5, for instance).

Further, comparing with approximation by ridgelets \cite{candes1998ridgelets}, we see from the results of Barron \cite{barron1993universal} and section 4.2 in \cite{candes1998ridgelets} that the space $B^s(\Omega)$ is contained in $R_{1,\infty}^{1+k+(d-1)/2}$ if $s \geq k + 1$. Thus, using ridgelets one can obtain a rate of $O(n^{-\frac{1}{2}-\frac{2s-1}{2d}})$, which is not quite as good as the rate attainable using cosine networks. We wish to emphasize that this approximation rate is not entirely trivial since we are considering the full spectral Barron space $B^s(\Omega)$ and not the space with frequencies restricted to a lattice, as in \cite{lu2021priori,candes1998ridgelets}.

Next, we consider the problem with activation function $\sigma(x) = \max(0,x)^k$. In Theorems \ref{piecewise-poly-approx-theorem} and \ref{high-smoothness-approximation} we show that in this case the convergence rate in $H^m(\Omega)$ can also be continuously improved with increasing $s$, up to the limit of 
\begin{equation}\label{best-rate-polynomials}
\|f-f_n\|_{H^m(\Omega)} \lesssim \|f\|_{\mathcal{B}^s}n^{m-(k+1)}\log{n},
\end{equation}
which is achieved when $s > (d+1)(k-m+\frac{1}{2})+m+\frac{1}{2}$. This maximal rate is the same as the rate attained by piecewise polynomials of degree $k$ in dimension $d=1$, and is significantly higher than the maximal rate obtained in \cite{petrushev1998approximation} in this case. In particular, this shows that regardless of the dimension $d$, for smooth enough functions $f$ (here the necessary amount of smoothness depends upon $d$), neural networks of the form \eqref{ridge-superposition} attain approximation rates which match the best possible rates in one dimension. 

High-order approximation rates for deeper networks have been studied in \cite{yarotsky2017error,yarotsky2018optimal,yarotsky2020phase}, however these results sometimes involve architectures which depend on the desired accuracy or even the function to be approximated. A theory of approximation by deeper networks in one dimension has also been developed in \cite{daubechies2019nonlinear}. In addition, approximation by deep networks with ReLU$^k$, or RePU, activation function has been studied in \cite{li2019better}. In contrast, our results apply already to shallow networks and show that high order approximation rates can be obtained for a class of sufficiently smooth functions even in high dimensions. A further interesting consequence, which is collected in Theorem \ref{improved-barron}, is that the approximation results obtained for sigmoidal activation functions by Barron \cite{barron1993universal} actually hold under weaker regularity conditions on the function $f$. In particular, instead of $f\in \mathcal{B}^1(\Omega)$ we only need $f\in \mathcal{B}^{\frac{1}{2}}(\Omega)$, although this comes at the cost of a constant which depends exponentially upon the dimension.

Compared with other results in the literature, the main significance of our results is that the smoothness of the activation function is fixed independently of the dimension. In particular, if the target class of functions is very smooth, say $f\in H^k(\Omega)$ with $k$ growing linearly in the dimension $d$, then finite elements \cite{brenner2007mathematical} of a sufficiently high degree or wavelets with a sufficiently high degree of smoothness can attain approximation rates which do not depend upon $d$ (since $k$ grows with $d$) \cite{devore1998nonlinear,daubechies1992ten}. For a space of functions perhaps more comparable to the spectral Barron class, we may also consider the ridgelet space $R_{p,q}^s$ for large $s$ depending linearly upon the dimension. In this case as well, dimension independent (again since $s$ depends upon $d$) can be obtained using ridgelets \cite{candes1998ridgelets}.

Another method which is particularly effective for high-dimensional problems is the sparse grids method \cite{bungartz2004sparse}. The sparse grid method approximates functions $f$ from the class 
\begin{equation}\label{sparse-grids-space}
\|f\|^2_{H^{k+1}_{mix}}:= \sum_{|\alpha|_\infty \leq k+1}\int_{\Omega} \left|\frac{\partial^{|\alpha|_1}}{\partial \alpha_1\cdots \partial\alpha_n} f(x)\right|^2dx < \infty
\end{equation}
by linear combinations of tensor products of piecewise degree $k$ polynomials on one-dimensinal grids with different resolutions in each coordinate direction. As such, the sparse grid method approximates $f$ with a piecewise polynomial function of degree $kd$. Using this method, an approximation rate in $L^2(\Omega)$ of
\begin{equation}\label{sparse-grids-estimate}
 \|f-f_n\|_{L^2(\Omega)} \lesssim n^{-(k+1)}(\log{n})^{(k+2)(d-1)}
\end{equation}
can be attained, where $n$ is the number of degrees of freedom \cite{bungartz2004sparse}. Thus, using polynomials of degree $kd$, the sparse grids method is able to attain similar approximation rates as shallow neural networks under high order smoothness assumptions. One additional potential advantage of the shallow networks is that the spectral Barron space is isotropic, while the sparse grids space is not.

However, in all of these examples the degree of the polynomials or the smoothness of the wavelets and ridgelets must grow with the dimension, which may be inconvenient, since for instance constructing conforming finite element spaces of high degree is a difficult problem \cite{lai2007spline,vzenivsek1970interpolation,argyris1968tuba,bramble1970triangular}. In contrast, for our results the activation function $\sigma$ is fixed independently of the dimension and so we are able to achieve these approximation rates using piecewise polynomials of a fixed degree. The reason this is possible is that the number of components of the piecewise polynomial represented by an expression of the form \eqref{ridge-superposition} grows as $n^d$ (the number of components obtained by cutting $\mathbb{R}^d$ by $n$ hyperplanes) while the number of parameters grows as $dn$. This suggests that perhaps (shallow) neural networks are particularly effective at approximating highly smooth functions in high dimensions.

This is the sense in which our results show that shallow neural networks overcome the ``curse of dimensionality''. It is by enabling the use of fixed degree polynomials even in high dimensions. The space of functions which we approximate, $\mathcal{B}^s(\Omega)$, is indeed very small and consists of highly smooth functions. 

However, in high dimensions $d$ the metric entropy of classical smoothness spaces with smoothness degree $s$ decays very slowly, specifically like $O(n^{-\frac{s}{d}})$ (see for instance \cite{lorentz1996constructive}, chapter 15). Consequently this class of functions fundamentally cannot be approximated efficiently in high-dimensions. To overcome this, many learning and approximation methods consider classes of functions whose smoothness either grows with dimension or takes a non-standard form. Examples of such spaces include the sparse grids spaces $H^{k+1}_{mix}$, the reproducting kernel spaces associated with kernel methods \cite{bengio2006curse}, the ridgelet spaces \cite{candes1998ridgelets}, and convex hulls of dictionaries. In this sense, the curse of dimensionality is overcome in a similar manner to other methods, namely by considering sufficiently smooth functions.

The preceding analysis suggests that the adaptive nature of the grid underlying a ReLU$^k$ network allows for a significantly better approximation rate than existing linear methods based on fixed grids. This suggests the potential of using such spaces for the solution of differential equations, which has been investigated in \cite{CiCP-28-1707}. In light of these results, the space of functions represented by shallow ReLU$^k$ networks may also be useful in understanding non-linear approximation by piecewise polynomials more generally, which has been a challenging problem \cite{devore1998nonlinear}.

The paper is organized as follows. In the next section, we consider approximation by networks with a cosine activation function. In addition to being of independent interest, key results in this section concerning the representation of functions in $\mathcal{B}^s(\Omega)$ will be used throughout the paper. Then, in section \ref{relu-result-section} we prove the main results concerning approximation by ReLU$^k$ networks. In section \ref{lower-bounds-section} we show that the results obtained in the previous sections are optimal. Finally, we give concluding remarks and further research directions.

\section{Approximation Rates for Cosine Networks}
To begin, we remark that throughout this manuscript, we use the following convention for the Fourier transform
\begin{equation}
 \hat{f}(\xi) = \int_{\mathbb{R}^d} f(x)e^{-2\pi \iu \xi\cdot x}dx,
\end{equation}
for which the inverse transform is given by
\begin{equation}
 f(x) = \int_{\mathbb{R}^d} \hat{f}(\xi)e^{2\pi \iu    \xi\cdot x}d\xi.
\end{equation}
We find that this convention results in the cleanest arguments, avoiding the necessity to keep track of normalizing constants.

In this section, we analyze the approximation properties of networks with a cosine activation function on the spectral Barron space $\mathcal{B}^s(\Omega)$. Specifically, consider approximating a function $f\in \mathcal{B}^s(\Omega)$ by a superposition of finitely many complex exponentials with coefficients that are bounded in $\ell^1$, i.e. by an element of the set
\begin{equation}
 \Sigma_{n,M} = \left\{\sum_{j=1}^n a_je^{2\pi \iu  \theta_j\cdot x}:~\theta_j\in \mathbb{R}^d,~a_j\in\mathbb{C},~\sum_{i=1}^n|a_i|\leq M\right\}.
\end{equation}

Alternatively, one can view this as the set of neural networks with a single hidden layer containing $n$ neurons with activation function $\sigma(x) = e^{2\pi \iu  x}$, whose weights are bounded in $\ell^1$.

Equivalently, we can consider approximation by networks with a cosine activation function
\begin{equation}
  \Sigma^{\cos}_{n,M} = \left\{\sum_{i=1}^n a_i\cos(2\pi\theta_i\cdot x + b_i):~\theta_i\in \mathbb{R}^d,~b_i\in \mathbb{R},~\sum_{i=1}^n|a_i| \leq M\right\}.
 \end{equation}
This is because 
\begin{equation}
e^{2\pi \iu  \theta_i\cdot x} = \cos(2\pi\theta_i\cdot x) + i\cos\left(2\pi\theta_i\cdot x - \frac{\pi}{2}\right)\in \Sigma^{\cos}_{2,2}
\end{equation}
and
\begin{equation}
 \cos(2\pi\theta_i\cdot x) = \frac{1}{2}e^{2\pi \iu  \theta_i\cdot x} + \frac{1}{2}e^{-2\pi \iu  \theta_i\cdot x}\in \Sigma^d_{2,1}.
\end{equation}
Thus we have $\Sigma_{n,M}\subset \Sigma^{\cos}_{2n,2M}$ and $\Sigma^{\cos}_{n,M} \subset \Sigma_{2n,M}$ and so the rates obtained for both sets will be the same. In what follows, we consider $\Sigma_{n,M}$ for convenience in dealing with the Fourier transform.

Let us first state the following simple result that can be obtained by following a calculation given in Section 3 of \cite{johnson2015saddle}. 
\begin{lemma} Given $\alpha>1$, consider
 \begin{equation}\label{alpha-g}
  g(t) = \begin{cases} 
      e^{-(1-t^2)^{1 - \alpha}} & t\in (-1,1) \\
      0 & \text{otherwise}.
   \end{cases}
 \end{equation}
then there is a constant $c_\alpha$ such that
 \begin{equation}\label{eq_181}
  |\hat{g}(\xi)|\lesssim e^{-c_\alpha|\xi|^{1-\alpha^{-1}}},
 \end{equation}
\end{lemma}

We begin with a key lemma showing that we only need frequencies lying on a lattice to represent functions $f$ with decaying Fourier transform on a bounded set.
\begin{lemma}\label{fourier-representation-lemma-general}
 Let $\Omega=[0,1]^d$ and $\mu:\mathbb{R}^d\rightarrow \mathbb{R}_+$ be a continuous weight function. Suppose that $\mu$ satisfies the following conditions
 \begin{itemize}
  \item $\mu(\xi + \omega) \leq \mu(\xi)\mu(\omega)$
  \item There exists a $0 < \beta < 1$ and a $c > 0$ such that $\mu(\xi) \lesssim e^{c|\xi|^{\beta}}$.
 \end{itemize}

 Suppose that $f$ satisfies
 \begin{equation}
  \int_{\mathbb{R}^d} \mu(\xi)|\hat{f}(\xi)|d\xi = C_f < \infty.
 \end{equation}

 Then for any $L > 1$, there exists an $a\in L^{-1}[0,1]^d$ (which may depend on $f$) and coefficients $c_\xi$, such that for $x\in \Omega$
 \begin{equation}
  f(x) = \sum_{\xi\in L^{-1}\mathbb{Z}^d}c_\xi e^{2\pi \iu  (a+\xi)\cdot x}
 \end{equation}
 and
 \begin{equation}
  \sum_{\xi\in L^{-1}\mathbb{Z}^d} \mu(a+\xi)|c_\xi| \lesssim C_f.
 \end{equation}

\end{lemma}

Note that the suppressed constant in the above lemma only depends upon $d,\mu$ and $L$, but not on $f$ or $a$. Furthmore, we note that from the proof below it follows that the suppressed constant depends exponentially on the dimension $d$.

\begin{proof}
 Since by assumption $L > 1$, there exists an $\epsilon$ such that $\Omega\subset [0,L-2\epsilon]^d$. We begin by constructing a cutoff function $\phi_\Omega$, which is identically $1$ on $\Omega$ and $0$ outside of $[-\epsilon, L-\epsilon]^d$. It will be important that the Fourier transform $\hat{\phi}_\Omega$ has sufficiently fast decay, so that
 \begin{equation}
  \int_{\mathbb{R}^d}\mu(\xi)|\hat{\phi}_\Omega(\xi)| < \infty.
 \end{equation}
 
 To construct this function, we follow closely the calculation made in \cite{johnson2015saddle}. Choose $\alpha > 1$ such that $\beta < 1-\alpha^{-1} < 1$ and consider the smooth one-dimensional bump function $g$ by \eqref{alpha-g}.
 Let $g_d$ denote the $n$-dimensional function
 \begin{equation}
  g_d(x) = \frac{1}{C}\prod_{i=1}^d g(x_i),
 \end{equation}
 where the normalization constant $C$ is chosen so that $\int_{\mathbb{R}^d} g_d(x) = 1$. Then by \eqref{eq_181} we see that
 \begin{equation}
  |\hat{g}_d(\xi)|\lesssim e^{-c_\alpha\sum_{i=1}^d|\xi_i|^{1-\alpha^{-1}}} \lesssim e^{-c_{\alpha,d}\alpha|\xi|^{1-\alpha^{-1}}},
 \end{equation}
 for a new constant $c_{\alpha,d}$.
 
 Finally, let $\Omega^\prime = [-\frac{\epsilon}{2}, L - \frac{3\epsilon}{2}]^d$ and define
 \begin{equation}
  \phi_\Omega = \left(4^d\epsilon^{-d}g_d\left(4\epsilon^{-1}x\right)\right)*\chi_{\Omega^\prime}(x).
 \end{equation}
 The compact support and normalization of $g_d$  implies that $\phi_\Omega|_\Omega = 1$ and $\phi_\Omega = 0$ outside of $[-\epsilon, L-\epsilon]^d$. Furthermore, we calculate
 \begin{equation}
  |\hat{\phi}_\Omega(\xi)| = \left|\hat{g}_d\left(\frac{\epsilon}{4}\xi\right)\widehat{\chi}_{\Omega^\prime}\right| \lesssim e^{-c_{\alpha,\Omega}|\xi|^{1-\alpha^{-1}}}
 \end{equation}
 for a constant $c_{\alpha,\Omega}$, since $\widehat{\chi}_{\Omega^\prime}$ is bounded. The growth condition on $\mu$, combined with $\beta < 1-\alpha^{-1} < 1$ means that
 \begin{equation}\label{eq_205}
  \int_{\mathbb{R}^d}\mu(\xi)|\hat{\phi}_\Omega(\xi)| < \infty.
 \end{equation}
 
 Now consider the function $h_f = \phi_\Omega f$. Evidently $h_f = f$ on $\Omega$ and $h_f$ is supported on $[-\epsilon,L-\epsilon]^d$. Notice further that $\hat{h}_f = \hat{\phi}_\Omega * \hat{f}$ and we calculate
 \begin{equation}
  \begin{split}
  \int_{\mathbb{R}^d} \mu(\xi)|\hat{h}_f(\xi)|d\xi &\leq \int_{\mathbb{R}^d} \int_{\mathbb{R}^d}\mu(\xi) |\hat{\phi}_\Omega(\xi - \omega)||\hat{f}(\omega)| d\omega d\xi \\
  &=\int_{\mathbb{R}^d} \left(\int_{\mathbb{R}^d}\mu(\xi + \omega) |\hat{\phi}_\Omega(\xi)|d\xi\right)|\hat{f}(\omega)| d\omega. 
  \end{split}
 \end{equation}
 Now $\mu(\xi + \omega) \leq \mu(\xi)\mu(\omega)$, so that we get
 \begin{equation}\label{eq_369}
  \int_{\mathbb{R}^d} \mu(\xi)|\hat{h}_f(\xi)|d\xi \leq \left(\int_{\mathbb{R}^d}\mu(\xi) |\hat{\phi}_\Omega(\xi)|d\xi\right)\left(\int_{\mathbb{R}^d}\mu(\omega) |\hat{f}(\omega)|d\omega\right) \lesssim C_f,
 \end{equation}
 where the implied constant depends the value of the integral in \eqref{eq_205}.
 
 We now rewrite the integral in \eqref{eq_369} as
 \begin{equation}
  \int_{\mathbb{R}^d} \mu(\xi)|\hat{h}_f(\xi)|d\xi = \int_{[0,L^{-1}]^d} \left(\sum_{\xi\in L^{-1}\mathbb{Z}^d} \mu(a+\xi)|\hat{h}_f(a+\xi)|\right)da \lesssim C_f.
 \end{equation}
 Certainly this means that there must exist an $a\in [0,L^{-1}]^d$ (depending on $f$) such that
 \begin{equation}
  \left(\sum_{\xi\in L^{-1}\mathbb{Z}^d} \mu(a+\xi)|\hat{h}_f(a+\xi)|\right) \lesssim C_f,
 \end{equation}
 where the implied constant only depends upon $L$ and $d$.
 
 We proceed to apply the Poisson summation formula and the fact that $h_f$ is supported in $[-\epsilon,L-\epsilon]^d$ to conclude that for a.e. $x\in \Omega \subset [-\epsilon,L-\epsilon]^d$ we have
 \begin{equation}
  f(x) = h_f(x) = \sum_{\nu \in L\mathbb{Z}^d} h_f(x+\nu)e^{2\pi \iu  a\cdot \nu} = \sum_{\xi\in L^{-1}\mathbb{Z}^d}\hat{h}_f(a+\xi) e^{2\pi \iu  (a+\xi)\cdot x}. 
 \end{equation}
 Here we have applied the Poisson summation formula to the function $g(\nu) = h_f(x+\nu)e^{2\pi \iu  a\cdot \nu}$, whose Fourier transform is easily seen to be $\hat{g}(\xi) = \hat{h}_f(a+\xi) e^{2\pi \iu  (a+\xi)\cdot x}$.
 
 Setting $c_\xi = \hat{h}_f(a+\xi)$ we obtain the desired result.
\end{proof}

We now apply Lemma \ref{fourier-representation-lemma-general} with $\mu(\xi) = (1 + |\xi|)^s$ to obtain the following corollary concerning the spectral Barron space $\mathcal{B}^s(\Omega)$.
\begin{corollary}\label{fourier-representation-lemma}
 Let $\Omega = [0,1]^d$ and $s \geq 0$. Let $f\in \mathcal{B}^s(\Omega)$. Then for any $L > 1$, there exists an $a\in L^{-1}[0,1]^d$ (potentially depending upon $f$) and coefficients $c_\xi$ such that
 for $x\in \Omega$
 \begin{equation}
  f(x) = \sum_{\xi\in L^{-1}\mathbb{Z}^d}c_\xi e^{2\pi \iu  (a+\xi)\cdot x}
 \end{equation}
 and
 \begin{equation}
  \sum_{\xi\in L^{-1}\mathbb{Z}^d} (1+|a+\xi|)^s|c_\xi| \lesssim \|f\|_{\mathcal{B}^s(\Omega)}.
 \end{equation}

\end{corollary}
\begin{proof}
 This follows immediately from Lemma \ref{fourier-representation-lemma-general} given the characterization of $\mathcal{B}^s(\Omega)$ and the elementary fact that $(1 + |\xi + \omega|)\leq (1 + |\xi| + |\omega|) \leq (1 + |\xi|)(1 + |\omega|)$.
\end{proof}

Corollary \ref{fourier-representation-lemma} can be used to improve upon the $O(n^{-\frac{1}{2}})$ approximation rate of cosine networks obtained in \cite{jones1992simple} when $f\in \mathcal{B}^s(\Omega)$ for $s > 0$.

\begin{theorem}\label{approximation-rate-theorem}
 Let $\Omega = [0,1]^d$, $0\leq m\leq s$, and $f\in \mathcal{B}^s(\Omega)$. Then there is an $M\lesssim \|f\|_{\mathcal{B}^s(\Omega)}$ such that
 \begin{equation}
  \inf_{f_n\in \Sigma_{n,M}} \|f-f_n\|_{H^m(\Omega)} \lesssim \|f\|_{\mathcal{B}^s(\Omega)}n^{-\frac{1}{2} - \frac{s-m}{d}}.
 \end{equation}

\end{theorem}

Note that the implied constant in the above theorem depends only upon $s,m$, and $d$, but not on $f$. Comparing this with the results in \cite{jones1992simple}, we obtain a dimension dependent improvement similar to what can be obtained using stratified sampling \cite{klusowski2018approximation} for rectified linear networks. However, the improvement in Theorem \ref{approximation-rate-theorem}, which is obtained via an entirely different argument, is greater and holds for cosine networks. We will consider rectified linear networks in the next section. Also, we note that for the Sobolev spaces $H^{\frac{d}{2} + s}(\Omega)$, this result already appears in \cite{petrushev1998approximation} . However, our results apply to the spectral Barron space $\mathcal{B}^s(\Omega)$, which is not quite comparable, but we do have $H^{\frac{d}{2} + s + \epsilon}(\Omega)\subset \mathcal{B}^s(\Omega)$ (see \cite{CiCP-28-1707} Lemma 2.5, for instance). Finally, as shown in Theorem \ref{fourier-lower-bound}, the rate in Theorem \ref{approximation-rate-theorem} is actually sharp.

\begin{proof}
 Choose $L > 1$. Note that all of the implied constants in what follows depend only upon $s,m,d$ and $L$, but not upon $f$. 
 
 By Corollary \ref{fourier-representation-lemma}, there exists an $a\in L^{-1}[0,1]^d$ and coefficients $c_\xi$ such that
  \begin{equation}\label{eq-580}
  f(x) = \sum_{\xi\in L^{-1}\mathbb{Z}^d}c_\xi e^{2\pi \iu  (a+\xi)\cdot x},
 \end{equation}
 and (here the first estimate follows since $|a|\leq L^{-d}\sqrt{d}$)
 \begin{equation}\label{eq_613}
  \sum_{\xi\in L^{-1}\mathbb{Z}^d} (1+|\xi|)^s|c_\xi| \eqsim \sum_{\xi\in L^{-1}\mathbb{Z}^d} (1+|a+\xi|)^s|c_\xi|  \lesssim \|f\|_{\mathcal{B}^s(\Omega)}.
 \end{equation}
 Consider the slightly enlarged set $\Omega^\prime = [0,L]^d \supset \Omega$. On this larger set, we have for $\xi\neq\nu\in L^{-1}\mathbb{Z}^d$
 \begin{equation}\label{orthogonality-condition}
  \langle e^{2\pi \iu  (a+\xi)\cdot x}, e^{2\pi \iu  (a+\nu)\cdot x}\rangle_{H^k(\Omega^\prime)} = 0,
 \end{equation}
 so that the frequencies in the expansion \eqref{eq-580} form an orthogonal basis in $H^k(\Omega^\prime)$. Moreover, their lengths satisfy
 \begin{equation}\label{length-estimate}
  \|e^{2\pi \iu  (a+\xi)\cdot x}\|_{H^m(\Omega^\prime)} \lesssim (1+|a+\xi|)^m \eqsim (1+|\xi|)^m.
 \end{equation}
 Order the frequencies $\xi\in L^{-1}\mathbb{Z}^d$ such that
 \begin{equation}
   (1+|\xi_1|)^{2m-s}|c_{\xi_1}| \geq (1+|\xi_2|)^{2m-s}|c_{\xi_2}| \geq (1+|\xi_3|)^{2m-s}|c_{\xi_3}| \geq \cdots.
 \end{equation}

 For $n \geq 1$, let $S_n = \{\xi_1,\xi_2,...,\xi_n\}$ and set 
 \begin{equation}
  f_n = \sum_{\xi\in S_n}c_\xi e^{2\pi \iu  (a+\xi)\cdot x} \in \Sigma_{n,M}
 \end{equation}
 for $M\lesssim \|f\|_{\mathcal{B}^s}$ by \eqref{eq_613}.
 
 We now estimate, using \eqref{orthogonality-condition} and \eqref{length-estimate},
 \begin{equation}\label{eq_634}
 \begin{split}
  \|f - f_n\|^2_{H^m(\Omega^\prime)} &= \left\|\sum_{\xi\in S_n^c} c_\xi e^{2\pi \iu  (a+\xi)\cdot x} \right\|^2_{H^m(\Omega^\prime)} \\
  &= \sum_{\xi\in S_n^c} |c_\xi|^2\|e^{2\pi \iu  (a+\xi)\cdot x}\|_{H^m(\Omega^\prime)}^2 \\
  &\lesssim \sum_{\xi\in S_n^c} |c_\xi|^2(1+|\xi|)^{2m}.
  \end{split}
 \end{equation}
 Using Hoelder's inequality, we get
 \begin{equation}\label{eq_638}
  \sum_{\xi\in S_n^c} |c_\xi|^2(1+|\xi|)^{2m} \leq \left(\sup_{\xi\in S_n^c} |c_\xi|(1+|\xi|)^{2m-s} \right)\left(\sum_{\xi\in S_n^c}|c_\xi|(1+|\xi|)^{s}\right)
 \end{equation}
 By \eqref{eq_613}, the second term above is $\lesssim \|f\|_{\mathcal{B}^s(\Omega)}$. For the first term, we note that \eqref{eq_613} implies that
 \begin{equation}\label{eq_642}
  \sum_{\nu\in S_n} |c_\nu|(1+|\nu|)^{2m-s}(1+|\nu|)^{2(s-m)}\lesssim \|f\|_{\mathcal{B}^s(\Omega)}.
 \end{equation}
 Now, by the definition of $S_n$, we have for every $\nu\in S_n$
 \begin{equation}
  \left(\sup_{\xi\in S_n^c} |c_\xi|(1+|\xi|)^{2m-s}\right) \leq |c_\nu|(1+|\nu|)^{2m-s},
 \end{equation}
 so that
 \begin{equation}
 \begin{split}
  &\left(\sup_{\xi\in S_n^c} |c_\xi|(1+|\xi|)^{2m-s}\right)\sum_{\nu\in S_n}(1+|\nu|)^{2(s-m)}\\
  &\leq \sum_{\nu\in S_n} |c_\nu|(1+|\nu|)^{2m-s}(1+|\nu|)^{2(s-m)}.
  \end{split}
 \end{equation}
 By \eqref{eq_642}, we thus have
 \begin{equation}
  \left(\sup_{\xi\in S_n^c} |c_\xi|(1+|\xi|)^{2m-s}\right) \lesssim \|f\|_{\mathcal{B}^s(\Omega)}\left(\sum_{\nu\in S_n}(1+|\nu|)^{2(s-m)}\right)^{-1}.
 \end{equation}
 The sum $\sum_{\nu\in S_n}(1+|\nu|)^{2(s-m)}$ is over $n$ elements of the lattice $L^{-1}\mathbb{Z}^d$, from which it easily follows by comparison with an integral (see, for instance, \cite{erdos1989lattice}) that
 \begin{equation}
  \sum_{\nu\in S_n}(1+|\nu|)^{2(s-m)} \gtrsim n^{1+\frac{2(s-m)}{d}},
 \end{equation}
 and we obtain
 \begin{equation}
  \left(\sup_{\xi\in S_n^c} |c_\xi|(1+|\xi|)^{2m-s}\right) \lesssim \|f\|_{\mathcal{B}^s(\Omega)}n^{-1-\frac{2(s-m)}{d}}.
 \end{equation}
 Combining this with \eqref{eq_634} and \eqref{eq_638}, we get
 \begin{equation}
  \|f - f_n\|^2_{H^m(\Omega^\prime)} \lesssim \|f\|^2_{\mathcal{B}^s(\Omega)}n^{-1-\frac{2(s-m)}{d}}.
 \end{equation}
 Finally, since $\Omega^\prime \supset \Omega$, we get
 \begin{equation}
  \|f - f_n\|_{H^m(\Omega)} \leq \|f - f_n\|_{H^m(\Omega^\prime)} \lesssim \|f\|_{\mathcal{B}^s(\Omega)}n^{-\frac{1}{2}-\frac{(s-m)}{d}},
 \end{equation}
 which completes the proof.
\end{proof}

In Theorem \ref{approximation-rate-theorem}, we obtained arbitrarily high polynomial rates of convergence for sufficiently smooth functions. Next we generalize this by showing that if the Fourier transform decays at a superpolynomial rate, then we can obtain spectral (i.e. superpolynomial) convergence as well. We begin by introducing an exponential version of the spectral Barron spaces.
\begin{definition}
Let $\Omega\subset \mathbb{R}^d$ be a bounded domain and let $0 < \beta < 1$ and $c > 0$. The exponential spectral Barron space with parameters $\beta$ and $c$ is defined by
\begin{equation}
 \mathcal{B}_{\beta,c}(\Omega):=\left\{f:\Omega\rightarrow\mathbb{R}:\|f\|_{\mathcal{B}_{\beta,c}(\Omega)}:=\inf_{f_e|\Omega = f}\int_{\mathbb{R}^d}e^{c|\xi|^\beta}|\hat{f}_e(\xi)|d\xi < \infty\right\},
\end{equation}
where the infimum is taken over all extension $f_e\in L^1(\mathbb{R}^d)$.
\end{definition}
 The space $\mathcal{B}_{\beta,c}(\Omega)$ is quite restrictive, however there it still contains a relatively large class of functions. For example, it contains satisfied by any linear combination of Gaussians or any band-limited function, i.e. any function whose Fourier transform is compactly supported.

For elements of $\mathcal{B}_{\beta,c}(\Omega)$, we can prove a superpolynomial convergence rate.
\begin{theorem}\label{spectral-convergence-theorem}
 Let $\Omega = [0,1]^d$, $0 < \beta < 1$, and $c > 0$.
 Then for any $m \geq 0$, there exists a $c^\prime > 0$ such that for $f\in \mathcal{B}_{\beta,c}(\Omega)$ and $M\lesssim \|f\|_{\mathcal{B}_{\beta,c}(\Omega)}$ we have
  \begin{equation}
  \inf_{f_n\in \Sigma_{n,M}} \|f-f_n\|_{H^m(\Omega)} \lesssim \|f\|_{\mathcal{B}_{\beta,c}(\Omega)}e^{-c^\prime n^{d^{-1}\beta}}.
 \end{equation}
\end{theorem}
Note that in this theorem the implied constant and the constant $c^\prime$ only depend upon $\beta,c,d$ and $m$, but not on $f$ or $n$.
\begin{proof}
 We use a similar argument to the proof of Theorem \ref{approximation-rate-theorem}. First, we apply Lemma \ref{fourier-representation-lemma-general} to the weight $\mu(\xi) = e^{c|\xi|^{\beta}}$, to obtain an $a\in L^{-1}[0,1]^d$ and coefficients $c_\xi$ such that
  \begin{equation}
  f(x) = \sum_{\xi\in L^{-1}\mathbb{Z}^d}c_\xi e^{2\pi \iu  (a+\xi)\cdot x}
 \end{equation}
 and
 \begin{equation}\label{eq_361}
  \sum_{\xi\in L^{-1}\mathbb{Z}^d}  e^{c|a+\xi|^{\beta}}|c_\xi| \lesssim \|f\|_{\mathcal{B}_{\beta,c}(\Omega)}.
 \end{equation}
 As in the proof of Theorem \ref{approximation-rate-theorem}, we note that the frequencies $e^{2\pi \iu  (a+\xi)\cdot x}$ are orthogonal on the enlarger set $\Omega^\prime = [0,L]^d$ and their norms are bounded by \eqref{length-estimate}.
 
 This time, we order the frequencies $\xi\in L^{-1}\mathbb{Z}^d$ such that
 \begin{equation}
 \begin{split}
  (1+|\xi_1|)^{2k}e^{-c|a+\xi_1|^{\beta}}|c_{\xi_1}| &\geq  (1+|\xi_2|)^{2k}e^{-c|a+\xi_2|^{\beta}}|c_{\xi_2}| \\
  &\geq  (1+|\xi_3|)^{2k}e^{-c|a+\xi_3|^{\beta}}|c_{\xi_3}| \geq \cdots.
  \end{split}
 \end{equation}
 Choosing $S_n = \{\xi_1,...,\xi_n\}$ and setting
 \begin{equation}
  f_n(x) = \sum_{\xi\in S_n}c_\xi e^{2\pi \iu  (a+\xi)\cdot x} \in \Sigma_{n,M},
 \end{equation}
 with $M\lesssim \|f\|_{\mathcal{B}_{\beta,c}(\Omega)}$, we obtain, using the argument between equations \eqref{eq_634} and \eqref{eq_638}, that
 \begin{equation}\label{eq_375}
   \|f - f_n\|^2_{H^m(\Omega^\prime)} \leq \left(\sup_{\xi\in S_n^c} |c_\xi|(1+|\xi|)^{2m}e^{-c|a+\xi|^{\beta}} \right)\left(\sum_{\xi\in S_n^c}|c_\xi|e^{c|a+\xi|^{\beta}}\right).
 \end{equation}
 By \eqref{eq_361}, the second factor is $\lesssim \|f\|_{\mathcal{B}_{\beta,c}(\Omega)}$.
 
 For the first factor, the argument between equations \eqref{eq_613} and \eqref{eq_642} implies that
 \begin{equation}\label{eq_381}
  \left(\sup_{\xi\in S_n^c} |c_\xi|(1+|\xi|)^{2m}e^{-c|a+\xi|^{\beta}} \right) \lesssim C_f\left(\sum_{\nu\in S_n} e^{2c|a+\nu|^{\beta}}(1 + |a + \nu|)^{-2m}\right)^{-1}.
 \end{equation}
 We now proceed to lower bound the sum on the right by considering its largest term. Since the sum is over $n$ elements of the lattice $a + L^{-1}\mathbb{Z}^d$, the longest vector, i.e. the largest length $|a+\xi|$ which occurs in the sum, must be $\gtrsim n^{\frac{1}{d}}$. In addition $(1 + |a + \xi|)^{2m} \lesssim e^{2\epsilon|a+\xi|^{\beta}}$ for any $\epsilon > 0$, so we see that there must exist a $c^\prime > 0$ such that
 \begin{equation}
  \left(\sum_{\nu\in S_n} e^{2c|a+\nu|^{\beta}}(1 + |a + \nu|)^{-2m}\right) \gtrsim e^{2c^\prime n^{\frac{\beta}{d}}}.
 \end{equation}
 Plugging this into \eqref{eq_381} and \eqref{eq_375} and using the fact that $\Omega^\prime \subset \Omega$, we get
 \begin{equation}
  \inf_{f_n\in \Sigma_{n,M}} \|f-f_n\|_{H^m(\Omega)} \lesssim \|f\|_{\mathcal{B}_{\beta,c}(\Omega)}e^{-c^\prime n^{\frac{\beta}{d}}},
 \end{equation}
 as desired.

\end{proof}

\section{Approximation Rates for ReLU$^k$ Networks}\label{relu-result-section}
In this section, we consider approximation by neural networks with activation function $$\sigma_k(x) = [\max(0,x)]^k$$ for $k=\mathbb{Z}_{\geq 0}$ (here we set $0^0 = 0$, i.e. $\sigma_0(x)$ is the Heaviside function). Specifically, we consider approximating a function $f$ by elements of the set
\begin{equation}
 \Sigma^k_{n} = \left\{\sum_{i=1}^n a_i\sigma_k(\omega_i\cdot x + b_i):~\omega_i\in S^{d-1},~b_i\in \mathbb{R},~a_i\in\mathbb{C}\right\},
\end{equation}
where we allow the coefficients $a_i$ to have arbitrarily large $\ell^1$-norm.

We will use Lemma \ref{fourier-representation-lemma} to obtain an improved approximation rate for such networks on the spectral Barron space $\mathcal{B}^m(\Omega)$. To do this, we introduce a multiscale approximation of the complex exponentials $e^{2\pi \iu x}$ using splines. We begin by recalling some facts about spline interpolation which will be important in the following analysis. We will refer to \cite{devore1993constructive} for most of this material. Note also that similar arguments have been used to study the approximation properties of neural networks in one dimension \cite{costarelli2015approximation,costarelli2013approximation}.

Instead of working directly with $\sigma_k$ it is more convenient to introduce the cardinal B-splines
\begin{equation}\label{card-b-splines-def}
 N_k(x) = \frac{1}{k!}\sum_{i=0}^{k+1}(-1)^i\binom{k+1}{i}\sigma_k(x-i) \in \Sigma_{k+2}^k,
\end{equation}
which are compactly supported on $[0,k+1]$. 

Let $\mathcal{S}^k_\lambda$ denote the Schoenberg space of piecewise degree $k$ splines on $\mathbb{R}$ with knots at $\lambda \mathbb{Z}$. It is well known that every spline $S\in\mathcal{S}^k_1$ can be written as
\begin{equation}\label{eq_963}
 S(x) = \sum_{j=-\infty}^\infty c_j(S)N_k(x-j),
\end{equation}
where $c_j$ are the de Boor-Fix functionals (see \cite{devore1993constructive}, section 5.3). Since the knots of the spline are all evenly spaced, the functionals $c_j(S)$ are all translations of the functional $c_0$, i.e.
\begin{equation}\label{eq_967}
 c_j(S) = c_0(S(\cdot-j)).
\end{equation}
Moreover, consider change the spacing between the knots, i.e. consider $\mathcal{S}^k_\lambda$. Then, if $S\in \mathcal{S}^k_\lambda$, $S(\lambda x)\in \mathcal{S}^k_1$ and equations \eqref{eq_963} and \eqref{eq_967} imply that
\begin{equation}
 S(\lambda x) = \sum_{j=-\infty}^\infty c_j(S(\lambda\cdot))N_k(x-j),
\end{equation}
so that
\begin{equation}
 S(x) = \sum_{j=-\infty}^\infty c_{j,\lambda}(S)N_k(\lambda^{-1}x-j),
\end{equation}
where the functionals $c_{j,\lambda}$ are given by $c_{j,\lambda}(S) = c_0(S(\lambda(\cdot - j)))$.

Now, we see from \cite{devore1993constructive}, Lemma 4.1 of Chapter 5, that
\begin{equation}\label{eq_983}
 |c_{j,\lambda}(S)| \leq C\|S\|_{L^\infty([\lambda j, \lambda (j+k+1)])},
\end{equation}
for a fixed constant $C$. Thus, by the Hahn-Banach theorem, we can extend the de Boor-Fix functionals $c_{j,\lambda}$ to functionals $\gamma_{j,\lambda}$ on $L^\infty([\lambda j, \lambda (j+k+1)])$ which satisfy the same bound. This allows us to define the quasi-interpolation operators
\begin{equation}\label{quasi-interpolation}
 Q_\lambda(f) = \sum_{j=-\infty}^\infty \gamma_{j,\lambda}(f)N_k(\lambda^{-1}x-j),
\end{equation}
which are bounded in $L^\infty$ (uniformly in $\lambda$) and satisfy $Q(S) = S$ for all splines $S\in \mathcal{S}^k_\lambda$ (see \cite{devore1993constructive}, section 5.4). Note that here and in what follows, we suppress the dependence on $k$ of the operators $Q_\lambda$ and the de Boor-Fix functions $\gamma_{j,\lambda}$ to simplify notation.

We are now in the position to introduce the following multiscale piecewise degree $k$ approximation to $e^{2\pi \iu   x}$. We write
\begin{equation}\label{multiscale-approx}
 e^{2\pi \iu  x} = Q_{2^{-1}}(e^{2\pi \iu   x}) + \sum_{l=2}^\infty [Q_{2^{-l}}(e^{2\pi \iu   x}) - Q_{2^{-(l-1)}}(e^{2\pi \iu   x})] = \sum_{l=1}^\infty h_l(x),
\end{equation}
where
\begin{equation}
 h_l(x) = Q_{2^{-l}}(e^{2\pi \iu   x}) - Q_{2^{-(l-1)}}(e^{2\pi \iu   x}),
\end{equation}
for $l > 1$ and $h_1(x) = Q_{2^{-1}}(e^{2\pi \iu   x})$. Since we clearly have $\mathcal{S}^k_{2^{-(l-1)}} \subset \mathcal{S}^k_{2^{-l}}$, we see that
$$Q_{2^{-l}}(Q_{2^{-(l-1)}}(e^{2\pi \iu   x})) = Q_{2^{-(l-1)}}(e^{2\pi \iu   x}),$$ 
so that we can rewrite $h_l$ as

\begin{equation}\label{eq_995}
h_l(x) = Q_{2^{-l}}\left(e^{2\pi \iu   x} - Q_{2^{-(l-1)}}(e^{2\pi \iu   x})\right) = Q_{2^{-l}}(e_{l-1}(x)) = \sum_{j=-\infty}^\infty \alpha_{j,l}N_k(2^{l}x-j),
\end{equation}
where the error $e_{l-1}$ is given by $e_{l-1}(x) = e^{2\pi \iu   x} - Q_{2^{-(l-1)}}(e^{2\pi \iu   x})$ and the coefficients $\alpha_{j,l}$ are given by $\alpha_{j,l} = \gamma_{j,2^{-l}}(e_{l-1})$.

We have the following lemma concerning this this piecewise degree $k$ approximation of $e^{2\pi \iu   x}$.
\begin{lemma}\label{multilevel-spline}
 The above expansion of $e^{2\pi \iu x}$ has the following properties.
 \begin{itemize}
  \item $\|e_l\|_{\infty} \lesssim 2^{-(k+1)l}$.
  \item The coefficients $\alpha_{j,l}$ in equation \eqref{eq_995} satisfy $|\alpha_{j,l}| \lesssim 2^{-(k+1)l}$.
  \item The series in \eqref{multiscale-approx} converges in $W^{m,\infty}(\mathbb{R})$ for $0 \leq m \leq k$.
 \end{itemize}
\end{lemma}
Note that the implied constants in the above lemma and the following proof only depend upon $k$ and not upon $l$ or $j$.
\begin{proof}
 The first statement follows immediately from Theorem 4.5 in \cite{devore1993constructive}, since $e_l(x) = e^{2\pi \iu   x} - Q_{2^{-l}}(e^{2\pi \iu   x})$ and $e^{2\pi \iu x}\in W^{k+1,\infty}(\mathbb{R})$.

 For the second statement, we note that
 \begin{equation}
  |\alpha_{j,l}| = |\gamma_{j,2^{-l}}(e_{l-1})| \leq C\|e_{l-1}\|_{L^\infty(\mathbb{R})} \lesssim 2^{-(k+1)(l-1)} \lesssim 2^{-(k+1)l},
 \end{equation}
 where the first inequality is due to the fact that $\gamma_{j,2^{-j}}$ is a Hahn-Banach extension of the de Boor-Fix functional $c_{j,2^{-j}}$ which satisfies \eqref{eq_983}.
 
 Finally, note that since $\|e_l\|_{L^\infty(\mathbb{R})} \lesssim 2^{-(k+1)l}\rightarrow 0$, we have that the series in \eqref{multiscale-approx} converges in $L^\infty(\mathbb{R})$ to $e^{2\pi \iu x}$. We now claim that
 \begin{equation}
  \|h_l\|_{W^{m,\infty}(\mathbb{R})} \lesssim 2^{-(k+1-m)l}.
 \end{equation}
 First, we note that simply by taking derivatives, we get
 \begin{equation}
  \|N_p(2^lx-j)\|_{W^{m,\infty}(\mathbb{R}, dx)} \lesssim 2^{ml}.
 \end{equation}
 Second, the B-splines $N_k(2^lx-j)$ are compactly supported and each point $x$ is covered by at most $p+1$ of them. Hence
 \begin{equation}
 \begin{split}
  \|h_l\|_{W^{m,\infty}} &= \left\|\sum_{j=-\infty}^\infty \alpha_{j,l}N_k(2^{l}x-j)\right\|_{W^{m,\infty}} \\
  &\leq (k+1)\sup_{j}|\alpha_{j,l}|\|N_k(2^lx-j)\|_{W^{m,\infty}(\mathbb{R}, dx)} \\
  &\lesssim 2^{-(k+1-m)l},
  \end{split}
 \end{equation}
 since $\alpha_{j,l}\lesssim 2^{-(k+1)l}$.
 
 This means that if $m\leq k$, then $\sum_{l=1}^\infty \|h_l\|_{W^{m,\infty}}$ is summable and hence the sum in \eqref{multiscale-approx} converges in $W^{m,\infty}(\mathbb{R})$. Clearly, its limit must be the same as the limit in $L^\infty$ and thus it converges to $e^{2\pi \iu x}$.
\end{proof}

Combining the multiscale expansion \eqref{multiscale-approx} with Lemma \ref{fourier-representation-lemma} and some ideas from \cite{makovoz1996random}, we obtain the following theorem.

\begin{theorem}\label{piecewise-poly-approx-theorem}
 Let $\Omega = [0,1]^d$ and $f\in \mathcal{B}^s(\Omega)$ for $s \geq \frac{1}{2}$. Let $k \in \mathbb{Z}_{\geq 0}$ and $m\geq 0$, with $m\leq s-\frac{1}{2}$ and $m < k + \frac{1}{2}$. Then for $n\geq 2$,
 \begin{equation}\label{bound-equation}
  \inf_{f_n\in \Sigma^k_{n}}\|f - f_n\|_{H^m(\Omega)} \lesssim \|f\|_{\mathcal{B}^s}n^{-t}\log(n)^q,
 \end{equation}
 where the exponent $t$ is given by
 \begin{equation}
 \begin{split}
 t &= \frac{1}{2} + \min\left(\frac{2(s-m)-1}{2(d+1)}, k-m+\frac{1}{2}\right) \\
 &= \begin{cases}
            \frac{1}{2}+\frac{2(s-m)-1}{2(d+1)} & \text{if}~s < (d+1)\left(k-m+\frac{1}{2}\right) + m + \frac{1}{2} \\
            k-m+1 & \text{if}~s \geq (d+1)\left(k-m+\frac{1}{2}\right) + m + \frac{1}{2}
            \end{cases}
\end{split}
\end{equation}
 and $q$ is given by $$q = \begin{cases}
            0 & \text{if}~s < (d+1)\left(k-m+\frac{1}{2}\right) + m + \frac{1}{2} \\
            1 & \text{if}~s > (d+1)\left(k-m+\frac{1}{2}\right) + m + \frac{1}{2}\\
            1 + (k-m+\frac{1}{2}) & \text{if}~s = (d+1)\left(k-m+\frac{1}{2}\right) + m + \frac{1}{2}
           \end{cases}.$$
\end{theorem}
Before beginning the proof, we remark that all of the implied constants in the $\eqsim$, $\gtrsim$, and $\lesssim$ can be seen to depend only on $s,k,m,d,L$ and $\delta$ ($L$ and $\delta$ chosen during the course of the proof), but not on $f$ or $n$. Further, we remark that the suppressed constant may depend exponentially on the dimension, i.e. as $A^d$ for some $A$. Finally, note that the maximal possible rate of $s-m+1$, which is achieved for sufficiently large $s$, is exactly the best achievable rate in one dimension. In Theorem \ref{relu-lower-bound} we use this fact to show that the rate of $s-m+1$ cannot be improved upon no matter how large $s$ is. It is an open problem whether such a rate can be obtained with less smoothness.

Comparing with other results in the literature, we see for instance that the results in \cite{klusowski2018approximation} apply to the cases $k=1,m=0$ (ReLU) and $k=2,m=0$ (ReLU$^2$). Furthermore, in \cite{CiCP-28-1707}, the general case $0\leq m\leq k$ is considered. In all of these cases the rate previously obtained was $O(n^{-\frac{1}{2}-\frac{1}{d}})$, while the rates in Theorem \ref{piecewise-poly-approx-theorem} are $O(n^{-\frac{1}{2}-\frac{2k+1}{2(d+1)}})$, which are significantly better for large $k$ and large $d$. However, the rates in Theorem \ref{piecewise-poly-approx-theorem} were obtained without the $\ell^1$-norm bound on the coefficients as in \cite{klusowski2018approximation} and \cite{CiCP-28-1707}. It is open whether the same rates can also be obtained with $\ell^1$-bounded  coefficients.

\begin{proof}
   Choose $L > 1$. Using Corollary \ref{fourier-representation-lemma}, we see that there exists an $a\in L^{-1}[0,1]^d$ and coefficients $a_\xi$ such that
 \begin{equation}\label{eq_675}
  f(x) = \sum_{\xi\in L^{-1}\mathbb{Z}^d} a_\xi (1+|a+\xi|)^{-s}e^{2\pi \iu  (a + \xi)\cdot x}
 \end{equation}
 and $\sum |a_\xi|\lesssim \|f\|_{\mathcal{B}^s}$. Here the suppressed constant depends potentially exponentially on the dimension, by the remarks in the previous section.
 
 We expand $e^{2\pi \iu   (a+\xi)\cdot x}$ using \eqref{multiscale-approx} to get
 \begin{equation}
  e^{2\pi \iu  (a + \xi)\cdot x} = \sum_{l=1}^\infty h_l((a + \xi)\cdot x),
 \end{equation}
 which holds in $W^{m,\infty}(\mathbb{R}^d)$ and thus in $H^m(\Omega)$ since $\Omega$ is bounded. 
 
 Expanding $h_l$ using equation \eqref{eq_995} and plugging this into equation \eqref{eq_675}, we obtain (in $H^m(\Omega)$)
 \begin{equation}
  f(x) = \sum_{\xi\in L^{-1}\mathbb{Z}^d}\sum_{l=1}^\infty \sum_{j=-\infty}^\infty a_\xi \alpha_{j,l} (1+|a+\xi|)^{-s}N_k(2^l(a+\xi)\cdot x - j).
 \end{equation}
 
 Now, since $x\in \Omega$, $\Omega$ is a bounded set, and $N_k$ is compactly supported, the number of non-zero terms in the inner-most sum above is finite. Indexing the values of $j$ for which $N_k(2^l(a+\xi)\cdot x - j)$ is non-zero for $x\in \Omega$ as $j_1,...,j_{n_{\xi, l}}$, we get
 \begin{equation}
  f(x) = \sum_{\xi\in L^{-1}\mathbb{Z}^d}\sum_{l=1}^\infty \sum_{p=1}^{n_{\xi,l}} a_\xi \alpha_{j_p,l} (1+|a+\xi|)^{-s} \psi_{\xi,l,p}(x),
 \end{equation}
 where
 \begin{equation}\label{psi-definition}
  \psi_{\xi,l,p}(x) = N_k(2^l(a+\xi)\cdot x - j_p).
 \end{equation}
 A straightforward calculation utilizing the compact support of $N_k$ implies that
 \begin{equation}\label{eq_454}
  \|\psi_{\xi,l,p}\|_{H^m(\Omega)} \lesssim 2^{l\left(m-\frac{1}{2}\right)}(1 + |\xi|)^{\left(m-\frac{1}{2}\right)}.
 \end{equation}
 Further, note that the number of terms $n_{\xi,l}$ satisfies
 \begin{equation}\label{eq_452}
  n_{\xi,l} \lesssim 2^l(1 + |\xi|).
 \end{equation}
This follows since for $x\in \Omega$, $y = 2^l(a + \xi)\cdot x$ takes on values in an interval of length at most $2^l|a + \xi|\text{diam}(\Omega)$ and $N_k$ has compact support of size $k+1$.

Let $\delta > 0$ to be specified later. We proceed to write
\begin{equation}\label{eq_458}
  f(x) = \sum_{\xi\in L^{-1}\mathbb{Z}^d}\sum_{l=1}^\infty \sum_{p=1}^{n_{\xi,l}} a_\xi 2^{-l(1+\delta)}(1+|\xi|)^{-1} \phi_{\xi,l,p}(x),
 \end{equation}
 where 
 \begin{equation}\label{eq_463}
  \phi_{\xi,l,p}(x) = 2^{l(1+\delta)}(1+|\xi|)\alpha_{j_p,l} (1+|a+\xi|)^{-s} \psi_{\xi,l,p}(x).
 \end{equation}
 Using \eqref{eq_463} and \eqref{eq_454} combined with the bound on $|\alpha_{j,l}|$ from Lemma \ref{multilevel-spline}, we calculate
 \begin{equation}\label{eq_470}
  \|\phi_{\xi,l,p}\|_{H^m(\Omega)} \lesssim 2^{-l(k-m+\frac{1}{2}-\delta)}(1+|\xi|)^{m-s+\frac{1}{2}}.
 \end{equation}
 
 We now observe that  by \eqref{eq_452} the $\ell^1$ norm of the coefficients of the $\phi_{\xi,l,p}$ in \eqref{eq_458} is bounded, namely
 \begin{equation}
 \begin{split}
  \sum_{\xi\in L^{-1}\mathbb{Z}^d}\sum_{l=1}^\infty \sum_{p=1}^{n_{\xi,l}} |a_\xi 2^{-l(1+\delta)}(1+|\xi|)^{-1}| &=
  \sum_{\xi\in L^{-1}\mathbb{Z}^d}|a_\xi|\sum_{l=1}^\infty n_{\xi,l}2^{-l(1+\delta)}(1+|\xi|)^{-1} \\
  &\lesssim \sum_{\xi\in L^{-1}\mathbb{Z}^d}|a_\xi|\sum_{l=1}^\infty 2^{-l\delta} \\
  & \lesssim \delta^{-1}\|f\|_{\mathcal{B}^m(\Omega)}.
  \end{split}
 \end{equation}
 We can now apply Theorem 1 in \cite{makovoz1996random} (note that this theorem still applies even though the coefficients in \eqref{eq_458} are potentially complex) to $f$ to conclude that there exists an
 \begin{equation}
  f_n = \sum_{i=1}^n a_i\phi_{\xi_i,l_i,p_i}(x)
 \end{equation}
 with $\sum_{i=1}^n|a_i| \lesssim \|f\|_{\mathcal{B}^s(\Omega)}$ such that
 \begin{equation}\label{bound_equation}
  \|f - f_n\|_{H^m(\Omega)} \lesssim \delta^{-1}\|f\|_{\mathcal{B}^s(\Omega)}\epsilon_n(\Phi)n^{-\frac{1}{2}},
 \end{equation}
 where $\Phi = \{\phi_{\xi,l,p}\}$ and $\epsilon_n(\Phi) = \inf\{\epsilon > 0:\Phi~\text{is covered by $n$ balls of diameter $\epsilon$}\}$ is the $n$-covering width of $\Phi$. 
 
 By choosing $\delta = k-m+\frac{1}{2} > 0$ we obtain the result at the endpoint $s = m+\frac{1}{2}$ (where the desired rate is $O(n^{-\frac{1}{2}})$) since by \eqref{eq_470} $$\epsilon_n(\Phi)\leq \epsilon_1(\Phi)\leq \sup \|\phi_{\xi,l,p}\|_{H^m(\Omega)}\lesssim 1.$$
 
 For larger $s$ we need to obtain a sharper bound on $\epsilon_n(\Phi)$. We do this by considering the covering number 
 \begin{equation}
  N_\Phi(\epsilon) = \min\{n:~\text{there is a covering of $\Phi$ by $n$ balls of diameter $\epsilon$}\},
 \end{equation}
 and noting that by definition $\epsilon_n(\Phi) = \inf\{\epsilon > 0: N_\Phi(\epsilon) \leq n\}$. 
 
 Given $\epsilon > 0$, we cover the set $\Phi$ by a single ball of radius $\frac{\epsilon}{2}$ centered at the origin, and cover each of the remaining elements with additional balls. This implies that
 \begin{equation}
  N_\Phi(\epsilon) \leq 1 + \left|\left\{\phi_{\xi,l,p}: \|\phi_{\xi,l,p}\|_{H^m(\Omega)} > \frac{\epsilon}{2}\right\}\right|.
 \end{equation}
 We proceed to count the number of $\phi_{\xi,l,p}$ with large norm. This process is messy but relatively straightforward.
 
 By \eqref{eq_470} we must count the indices $\xi\in L^{-1}\mathbb{Z}^d,l\in \mathbb{Z}_{>0}$ and $s=1,...,n_{\xi,l}$ for which
 \begin{equation}
  \epsilon \lesssim 2^{-l(k-m+\frac{1}{2}-\delta)}(1+|\xi|)^{m-s+\frac{1}{2}}.
 \end{equation}
 We observe that this condition implies that we must choose $\xi$ so that $(1+|\xi|)^{m-s+\frac{1}{2}} \gtrsim \epsilon$ and $l$ so that
 \begin{equation}\label{eq_506}
  2^{l(k-m+\frac{1}{2}-\delta)} \lesssim \epsilon^{-1}(1+|\xi|)^{m-s+\frac{1}{2}}.
 \end{equation}
 In addition, for each of these values of $\xi$ and $l$, we get $n_{\xi,l}\lesssim 2^l(1+|\xi|)$ differt values of $p$. Combining these observations, we see that
 \begin{equation}
  \left|\left\{\phi_{\xi,l,p}: \|\phi_{\xi,l,p}\|_{H^m(\Omega)} > \frac{\epsilon}{2}\right\}\right| \lesssim \sum_{\substack{\xi\in L^{-1}\mathbb{Z}^d\\|\xi|\leq R}} (1+|\xi|) \sum_{l\in L(\epsilon,\xi)} 2^l,
 \end{equation}
 where $R \lesssim \epsilon^{\frac{1}{m-s+\frac{1}{2}}}$ (note that here we require $m-s+\frac{1}{2} < 0$) and $L(\epsilon,\xi)$ consists of indices $l$ which satisfy \eqref{eq_506}. Taking a logarithm, the set $L(\epsilon,\xi)$ can be characterized by
 \begin{equation}
  l\leq \left(k-m+\frac{1}{2}-\delta\right)^{-1}\left(-\log(\epsilon) + \left(m-s+\frac{1}{2}\right)\log(1+|\xi|)\right) + C
 \end{equation}
 for some constant $C$. Using this bound on $l$, combined with the fact that $\sum_{l=1}^k 2^l \lesssim 2^k$, we get
 \begin{equation}
  \sum_{l\in L(\epsilon,\xi)} 2^l \lesssim \epsilon^{\frac{-1}{k-m+\frac{1}{2}-\delta}}(1+|\xi|)^{\frac{m-s+\frac{1}{2}}{k-m+\frac{1}{2}-\delta}}.
 \end{equation}
 So we get
 \begin{equation}\label{eq_524}
  \left|\left\{\phi_{\xi,l,p}: \|\phi_{\xi,l,p}\|_{H^m(\Omega)} > \frac{\epsilon}{2}\right\}\right| \lesssim \epsilon^{\frac{-1}{k-m+\frac{1}{2}-\delta}} \sum_{\substack{\xi\in L^{-1}\mathbb{Z}^d\\|\xi|\leq R}} (1+|\xi|)^{1 + \frac{m-s+\frac{1}{2}}{k-m+\frac{1}{2}-\delta}}.
 \end{equation}
 For the final sum, we distinguish between two cases. 
 
 First, if $m-s+\frac{1}{2} > -(d+1)\left(k-m+\frac{1}{2}\right)$, then we can choose a fixed $\delta = \delta(s,m,k,d) > 0$ small enough so that 
 \begin{equation}
  1 + \frac{m-s+\frac{1}{2}}{k-m+\frac{1}{2}-\delta} > -d.
 \end{equation}
 In this case, by comparing the sum over the lattice $L^{-1}\mathbb{Z}^d$ to an integral, the sum in \eqref{eq_524} satisfies
 \begin{equation}
  \sum_{\substack{\xi\in L^{-1}\mathbb{Z}^d\\|\xi|\leq R}} (1+|\xi|)^{1 + \frac{m-s+\frac{1}{2}}{k-m+\frac{1}{2}-\delta}} \lesssim R^{d+1+\frac{m-s+\frac{1}{2}}{k-m+\frac{1}{2}-\delta}} \lesssim \epsilon^{\frac{d+1}{m-s+\frac{1}{2}} + \frac{1}{k-m+\frac{1}{2}-\delta}},
 \end{equation}
 since $R\lesssim \epsilon^{\frac{1}{m-s+\frac{1}{2}}}$. Combining this with \eqref{eq_524} we get
 \begin{equation}
  \left|\left\{\phi_{\xi,l,p}: \|\phi_{\xi,l,p}\|_{H^m(\Omega)} > \frac{\epsilon}{2}\right\}\right| \lesssim\epsilon^{\frac{d+1}{m-s+\frac{1}{2}}}.
 \end{equation}
 This implies that for small $\epsilon$, $N_\Phi(\epsilon)\lesssim \epsilon^{\frac{d+1}{m-s+\frac{1}{2}}}$ and so
 \begin{equation}
  \epsilon_n(\Phi) \lesssim n^{\frac{m-s+\frac{1}{2}}{d+1}}.
 \end{equation}
 Plugging this into \eqref{bound_equation}, we get
 \begin{equation}\label{eq_549}
  \|f - f_n\|_{H^m(\Omega)} \lesssim \delta(s,m,k,d)^{-1}\|f\|_{\mathcal{B}^s(\Omega)}n^{\frac{m-s+\frac{1}{2}}{d+1}}n^{-\frac{1}{2}}\lesssim \|f\|_{\mathcal{B}^s(\Omega)}n^{-\frac{1}{2}-\frac{s-m-\frac{1}{2}}{d+1}}.
 \end{equation}

 Next, if $m-s+\frac{1}{2} \leq -(d+1)\left(k-m+\frac{1}{2}\right)$, then for any $\delta > 0$ we get
 \begin{equation}
  1 + \frac{m-s+\frac{1}{2}}{k-m+\frac{1}{2}-\delta} < -d.
 \end{equation}
 In this case, the sum in \eqref{eq_524} is summable and we get
 \begin{equation}
  \sum_{\substack{\xi\in L^{-1}\mathbb{Z}^d\\|\xi|\leq R}} (1+|\xi|)^{1 + \frac{m-s+\frac{1}{2}}{k-m+\frac{1}{2}-\delta}} \lesssim 1
 \end{equation}
 if $m-s+\frac{1}{2} < -(d+1)\left(k-m+\frac{1}{2}\right)$, and in the special case where $m-s+\frac{1}{2} = -(d+1)\left(k-m+\frac{1}{2}\right)$, we get
\begin{equation}
  \sum_{\substack{\xi\in L^{-1}\mathbb{Z}^d\\|\xi|\leq R}} (1+|\xi|)^{1 + \frac{m-s+\frac{1}{2}}{k-m+\frac{1}{2}-\delta}} \lesssim \delta^{-1}.
 \end{equation}
 Combining this with \eqref{eq_524} we get
 \begin{equation}
  \left|\left\{\phi_{\xi,l,p}: \|\phi_{\xi,l,p}\|_{H^m(\Omega)} > \frac{\epsilon}{2}\right\}\right| \lesssim \epsilon^{\frac{-1}{k-m+\frac{1}{2}-\delta}},
 \end{equation}
 where we need an extra factor of $\delta^{-1}$ in the special case where $m-s+\frac{1}{2} = -(d+1)\left(k-m+\frac{1}{2}\right)$.
 This implies that up to a factor of $\delta^{-(k-m+\frac{1}{2}-\delta)}$ in this special case, we have
 \begin{equation}
  \epsilon_n(\Phi) \lesssim n^{-(k-m+\frac{1}{2}-\delta)}.
 \end{equation}
 Using \eqref{bound_equation}, we get
 \begin{equation}
  \|f - f_n\|_{H^m(\Omega)} \lesssim \delta^{-1}\|f\|_{\mathcal{B}^s(\Omega)}n^{-(k-m+\frac{1}{2}-\delta)}n^{-\frac{1}{2}},
 \end{equation}
 where the power of $\delta$ is replaced by $-1-(k-m+\frac{1}{2}-\delta)$ in the endpoint case. Finally, optimizing over $\delta$, we get
 \begin{equation}\label{eq_580}
  \|f - f_n\|_{H^m(\Omega)} \lesssim \|f\|_{\mathcal{B}^s(\Omega)}n^{-(k-m+1)}\log(n),
 \end{equation}
 where in the endpoint case the logarithm is taken to the power $1 + (k-m+\frac{1}{2})$.
 
 Combining the results of \eqref{eq_580} and \eqref{eq_549} with the previously discussed result at $s = m + \frac{1}{2}$, we get
 \begin{equation}
  \|f - f_n\|_{H^m(\Omega)} \lesssim \|f\|_{\mathcal{B}^s(\Omega)}n^{-t}\log(n)^q,
 \end{equation}
 where $t = \min\left(\frac{1}{2}+\frac{s-m-\frac{1}{2}}{d+1}, k-m+1\right)$ and $q$ is given by
 \begin{equation}
 q = \begin{cases}
            0 & \text{if}~t < k-m+1 \\
            1 & \text{if}~t < \frac{1}{2}+\frac{2(s-m)-1}{2(d+1)}\\
            1 + (k-m+\frac{1}{2}) & \text{otherwise}
           \end{cases}.
 \end{equation}
 This completes the proof since \eqref{psi-definition}, \eqref{eq_463}, and \eqref{card-b-splines-def} imply that $\phi_{\xi,l,p}\in \Sigma_{k+2}^k$.

\end{proof}

In the case where $f$ is highly smooth, i.e. $f\in \mathcal{B}^s(\Omega)$ with $s > (d+1)(k-m-\frac{1}{2})+m+\frac{1}{2}$, we obtain, up to a logarithmic factor, an approximation rate of $O(n^{-k-1+m})$ in $H^m(\Omega)$. We state this as a separate theorem.
\begin{theorem}\label{high-smoothness-approximation}
 Let $\Omega = [0,1]^d$, $k \in \mathbb{Z}_{\geq 0}$ and $m\geq 0$, with $m < k + \frac{1}{2}$. Suppose that $f\in \mathcal{B}^s(\Omega)$ for $s$ sufficiently large, specifically $$s > (d+1)\left(k-m+\frac{1}{2}\right) + m + \frac{1}{2}.$$ Then for $n\geq 2$,
 \begin{equation}\label{bound-equation}
  \inf_{f_n\in \Sigma^k_{n}}\|f - f_n\|_{H^m(\Omega)} \lesssim \|f\|_{\mathcal{B}^s}n^{m-(k+1)}\log(n).
 \end{equation}
\end{theorem}
Further, the special case $s = \frac{1}{2}$, $m=0$ shows that the approximation rates obtained in \cite{barron1993universal} apply to a significantly larger class of functions if the $\ell^1$-coefficient bound on the neural network is dropped.
\begin{theorem}\label{improved-barron}
 Let $\Omega = [0,1]^d$ and $f\in B^{s}(\Omega)$ for $s = \frac{1}{2}$. Suppose that $\sigma$ is an arbitrary sigmoidal function. Then
 \begin{equation}
  \inf_{f_n\in \Sigma_n^\sigma}\|f - f_n\|_{L^2(\Omega)} \lesssim \|f\|_{\mathcal{B}^s}n^{-\frac{1}{2}},
 \end{equation}
 where
 \begin{equation}
  \Sigma_n^\sigma = \left\{\sum_{i=1}^n a_i\sigma(\omega_i\cdot x + b_i):~\omega_i\in \mathbb{R}^d,~b_i\in \mathbb{R},~a_i\in\mathbb{C}\right\}.
 \end{equation}

\end{theorem}
This result is obtained for $f\in B^1(\Omega)$ by Barron \cite{barron1993universal}. Here we show that in fact the condition $f\in B^{\frac{1}{2}}(\Omega)$ is sufficient.
\begin{proof}
 By Theorem \ref{piecewise-poly-approx-theorem}, the result holds if $\sigma = \sigma_0$ is the Heaviside function. For general sigmoidal $\sigma$, the result follows by noting that $\lim_{t\rightarrow \infty} \|\sigma(t(\omega\cdot x-b)) - \sigma_0(\omega\cdot x-b)\|_{L^2(\Omega)} \rightarrow 0$ for any $b$ and $\omega$.
\end{proof}

\section{Lower Bounds for Cosine Networks}\label{lower-bounds-section}
In this section, we derive lower bounds which complement Theorems \ref{approximation-rate-theorem} and \ref{piecewise-poly-approx-theorem}.
We begin with lower bounds on the approximation rate of cosine networks on $\mathcal{B}^s(\Omega)$. In particular, we show that the approximation rate of Theorem \ref{approximation-rate-theorem} cannot be substantially improved when $m = 0$, i.e. when we are approximating in $L^2(\Omega)$. We prove this even when the coefficients are only required to be bounded in $\ell^\infty$, i.e. when approximating from the set
\begin{equation}
 \Sigma^\infty_{n,M} = \left\{\sum_{j=1}^n a_je^{2\pi \iu  \theta_j\cdot x}:~\theta_j\in \mathbb{R}^d,~a_j\in\mathbb{C},~|a_i|\leq M\right\}.
\end{equation}
We have the following result.
\begin{theorem}\label{fourier-lower-bound}
 Let $\Omega = [0,1]^d$ and $s\geq 0$. Then we have
 \begin{equation}
  \limsup_{n\rightarrow \infty} \left[\sup_{\|f\|_{\mathcal{B}^s(\Omega)} \leq 1}\inf_{f_n\in\Sigma^\infty_{n,M}} \|f - f_n\|_{L^2(\Omega)}\right]n^{\frac{1}{2} + \frac{s}{d} + \epsilon} = \infty
 \end{equation}
 for any $M,\epsilon > 0$.
\end{theorem}

Lower bounds for the $\sigma_k$ activation function were obtained for $k=0$ obtained in \cite{makovoz1996random} and for $k\geq 1$ in \cite{klusowski2018approximation}. However, for $\sigma_k$ the lower bounds obtained do not match the best known rates. This gap has recently been closed in \cite{siegel2021sharp}. In contrast, Theorem \ref{approximation-rate-theorem} combined with Theorem \ref{fourier-lower-bound} gives the optimal approximation rate for cosine networks on the spectral Barron space $\mathcal{B}^s(\Omega)$. 

We also remark that a similar lower bound is obtained in section 7.2 of \cite{candes1998ridgelets}. However, the lower bound here is in terms of approximation by dictionary expansions where the size of the dictionary depends polynomially upon the number of terms. In this case the correct tool is to find large hypercubes with the given class \cite{donoho1995empirical}. In contrast, the result we prove applies to the infinite, even non-compact dictionary of Fourier modes and the tool used is the metric entropy. As a consequence, we must assume that the coefficients are bounded. We are not quite sure how to remove this assumption completely although it can be relaxed to a bound which grows polynomially with the number of terms $n$. 

\begin{proof}
 The argument is a modification of the methods in \cite{klusowski2018approximation,makovoz1996random}. The main new difficulty is in dealing with the non-compactness of the set $\{e^{2\pi\iu \theta\cdot x}:\theta\in \mathbb{R}^d\}\subset L^2(\Omega)$. 
 
 Suppose to the contrary that for some $M,\epsilon > 0$, we have
 \begin{equation}\label{eq_717}
  \sup_{\|f\|_{\mathcal{B}^s(\Omega)} \leq 1}\inf_{f_n\in\Sigma_{n,M}} \|f - f_n\|_{L^2(\Omega)} \lesssim n^{-\frac{1}{2} - \frac{s}{d} - \epsilon}.
 \end{equation}
 For $R > 0$ consider the set $$S(R) = \{\phi_\omega(x) := (1+|\omega|)^{-s}e^{2\pi \iu \omega\cdot x}:~\omega\in \mathbb{Z}^d,~|\omega|_\infty \leq R\}.$$
 In the following proof all of the implied constants are independent of $R$.
 
 The elements $\phi_\omega\in S(R)$ are orthogonal in $L^2(\Omega)$ and satisfy $\|\phi_\omega\|_{L^2(\Omega)} = (1+|\omega|)^{-s} \gtrsim R^{-s}$. In addition, it is clear that $\|\phi_\omega(x)\|_{\mathcal{B}^s(\Omega)} \leq 1$.
 
 We now make use of the following combinatorial fact which follows from Berge's theorem (see \cite{berge1973graphs,kainen1993quasiorthogonal}): given a set $S$ of size $n$, there exist at least $2^{cn}$ subsets of $S$ whose pairwise symmetric differences are at least $\frac{n}{4}$, where $c > 0$ is a universal constant (i.e. independent of $n$).
 
 We apply this to the set $S(R)$ to see that there are subsets $S_1,...,S_N\subset S(R)$ with $N = 2^{cR^d}$, such that for any $i\neq j$, we have $|S_i - S_j| \geq \frac{1}{4}R^d$. Consider the elements $\phi_i\in \mathcal{B}^s(\Omega)$ defined by
 \begin{equation}
  \phi_i(x) = \frac{1}{R^d} \sum_{\phi_\omega\in S_i}\phi_\omega(x).
 \end{equation}
 We clearly have $\|\phi_i(X)\|_{\mathcal{B}^s(\Omega)} \leq 1$. Moreover, since $|S_i - S_j| \geq \frac{1}{4}R^d$, $\|\phi_\omega\|_{L^2(\Omega)} \gtrsim R^{-s}$, and the $\phi_\omega$ are orthogonal, we see that for $i\neq j$
 \begin{equation}
 \begin{split}
 \|\phi_i(x) - \phi_j(x)\|_{L^2(\Omega)} &= \frac{1}{R^d}\left\|\sum_{\phi_\omega\in S_i-S_j}\phi_\omega(x)\right\|_{L^2(\Omega)} \\ &
 \gtrsim \frac{R^{-s}}{R^d}\sqrt{|S_i-S_j|} \geq \frac{R^{-s-\frac{d}{2}}}{2}.
 \end{split}
 \end{equation}
 Thus, we have at least $N = 2^{cR^d}$ elements $\phi_i$ which satisfy $\|\phi_i(x)\|_{\mathcal{B}^s(\Omega)} \leq 1$, and such that every pair differs by at least $\delta \gtrsim R^{-s-\frac{d}{2}}$ in $L^2(\Omega)$. Note that we could also have obtained this from \cite{klusowski2018approximation}, Lemma 1 in section 2.2.
 
 By \eqref{eq_717}
 there exist $\phi_{i,n}\in \Sigma_{n,M}$ which satisfy
 \begin{equation}
  \|\phi_{i,n} - \phi_i\|_{L^2(\Omega)} \leq \frac{\delta}{6},
 \end{equation}
 for an $n$ which satisfies 
 \begin{equation}\label{n-equation}
n \lesssim \delta^{-\frac{2d}{d+2s+2d\epsilon}}\lesssim R^{\frac{d(2s+d)}{2s+d+2d\epsilon}} = R^{d-t},
 \end{equation} 
 where $t(s,d,\epsilon) > 0$.
 
 Let $P_R$ denote the projection onto the space spanned by $S(R)$, i.e. onto the space spanned by the frequencies $e^{2\pi \iu \omega\cdot x}$ for $\omega\in \mathbb{Z}^d$, $|\omega|_\infty \leq R$. 
 Consider the projection $P_R(e^{2\pi \iu \theta\cdot x})$ for $\theta\in \mathbb{R}^d$. We calculate
 \begin{equation}\label{projection_calc}
 \begin{split}
  \|P_R(e^{2\pi \iu \theta\cdot x})\|^2_{L^2(\Omega)} &= \sum_{\substack{\omega\in \mathbb{Z}^d\\ |\omega|_\infty \leq R}} \left|\int_{[0,1]^d} e^{2\pi \iu (\theta - \omega)x}dx\right|^2 \\
  &\leq \frac{1}{(2\pi)^{2d}}\sum_{\substack{\omega\in \mathbb{Z}^d\\ |\omega|_\infty \leq R}}\prod_{i=1}^d\frac{1}{|\theta_i - \omega_i|^2}.
  \end{split}
 \end{equation}
 
 Choose $K$ large enough such that $\|P_R(e^{2\pi \iu \theta\cdot x})\|_{L^2(\Omega)} \leq \frac{\delta}{6nM}$ as long as $|\theta|_\infty \geq K$. By \eqref{projection_calc} and \eqref{n-equation}, this will be guaranteed if
 \begin{equation}
  K \geq R + \frac{6}{(2\pi)^d}\delta^{-1}MnR^{\frac{d}{2}} \lesssim R^{s+2d-t},
 \end{equation}
 and so we can choose $K\lesssim R^{s+2d-t}$.

 We proceed to truncate the $\phi_{i,n}\in \Sigma_{n,M}$ at frequencies with magnitude $K$. In particular, if
 \begin{equation}
  \phi_{i,n} = \sum_{j=1}^n a_{i,j}e^{2\pi \iu \theta_{i,j}\cdot x},
 \end{equation}
 we set $T^K_i = \{j:|\theta_{i,j}|_\infty \leq K\}$ and 
 \begin{equation}
 \phi^K_{i,n} = \sum_{j\in T^K_i} a_{i,j}e^{2\pi \iu \theta_{i,j}\cdot x}.
 \end{equation}
 Our choice of $K$ guarantees that
 \begin{equation}
  \|P_R(\phi_{i,n}^K) - P_R(\phi_{i,n})\|_{L^2(\Omega)} \leq \frac{\delta}{6},
 \end{equation}
 which implies that
 \begin{equation}
 \begin{split}
  \|P_R(\phi_{i,n}^K) - \phi_i\|_{L^2(\Omega)} &\leq \|P_R(\phi_{i,n}^K) - P_R(\phi_{i,n})\|_{L^2(\Omega)} + \|P_R(\phi_{i,n} - \phi_i)\|_{L^2(\Omega)} \\
  &\leq \frac{\delta}{6} + \frac{\delta}{6} = \frac{\delta}{3},
  \end{split}
 \end{equation}
 since $P_r(\phi_i) = \phi_i$.
 
 We now conclude that for $i\neq j$, we have
 \begin{equation}\label{eq_779}
 \begin{split}
  \|\phi_{i,n}^K - \phi_{j,n}^K\|_{L^2(\Omega)} &\geq \|P_R(\phi_{i,n}^K) - P_R(\phi_{j,n}^K)\|_{L^2(\Omega)} \\
  &\geq \|\phi_j - \phi_i\|_{L^2(\Omega)} - \|P_R(\phi_{j,n}^K) - \phi_j\|_{L^2(\Omega)} - \|P_R(\phi_{i,n}^K) - \phi_i\|_{L^2(\Omega)}\\
& \geq \delta - \frac{\delta}{3} - \frac{\delta}{3} = \frac{\delta}{3}.
\end{split}
 \end{equation}
 
 However, on the other hand, we calculate that
 \begin{equation}
  \|e^{2\pi \iu \theta_1\cdot x} - e^{2\pi \iu \theta_2\cdot x}\|^2_{L^2(\Omega)} = \int_{[0,1]^d} |1 - e^{2\pi \iu (\theta_1 - \theta_2)\cdot x}|^2dx \lesssim |\theta_1 - \theta_2|^2.
 \end{equation}
 We now cover the cube $C_K = \{\theta:|\theta|_\infty \leq K\}$ with $N_1$ frequencies $\nu_1,...,\nu_{N_1}$ such that for every $\theta\in C_K$, there exists an $i$ with 
 \begin{equation}
 \|e^{2\pi \iu \theta\cdot x} - e^{2\pi \iu \nu_i\cdot x}\|_{L^2(\Omega)} \leq \frac{\delta}{18nM}.
 \end{equation}
 By the above calculation, this can be done with
 \begin{equation}
  N_1\lesssim (KnM\delta^{-1})^d \lesssim R^{(2s+\frac{5}{2}d-t)d},
 \end{equation}
 where here we have taken into account the dependence of $K$, $n$ and $\delta$ on $R$.
 
 Further, we consider the cube $$A_M = \{\vec{a} = (a_1,...,a_n): |a_i|\leq M\},$$
 which we can cover with $N_2$ elements $\vec{a}_1,...,\vec{a}_{N_2}$ such that for every $\vec{a}\in A_M$, there is an index $i$ with $|\vec{a} - \vec{a}_i|_1 \leq \frac{\delta}{18}$. We can do this with
 \begin{equation}
  N_2 \lesssim (Mn\delta^{-1})^{2n}\lesssim M^{2R^{d-t}}R^{2(s+\frac{3d}{2}-t)R^{d-t}},
 \end{equation}
 where the $2n$ is because the components of $\vec{a}$ can be complex, we have expanded $\delta$ and $n$ in terms of $R$, and used that fact that if each component differs by $\delta/18n$, then the $\ell^1$-norm differs by at most $\delta/18$ as well.

 Given a
 \begin{equation}
 \phi^K_{i,n} = \sum_{j\in T^K_i} a_{i,j}e^{2\pi \iu \theta_{i,j}\cdot x},
 \end{equation}
 we proceed to perturb each of the $\theta_{i,j}$ to one of the frequencies $\nu_i$ and the coefficients $a_{i,j}$ to one of the $\vec{a}_j$. By the preceding analysis, we can thus land at one of
 \begin{equation}
  \bar{N} = N_2N_1^n \lesssim R^{(2s+\frac{5}{2}d-t)dR^{d-t}}M^{2R^{d-t}}R^{2(s+\frac{3d}{2}-t)R^{d-t}}
 \end{equation}
 elements by perturbing $\phi^K_{i,n}$ by at most $M\frac{\delta}{18M} + \frac{\delta}{18} = \frac{\delta}{9}$. Since for $i\neq j$, $\phi^K_{i,n}$ and $\phi^K_{j,n}$ differ by at least $\frac{\delta}{3}$ from \eqref{eq_779}, we see that they must all land at distinct elements after this perturbation. This implies that
 \begin{equation}
  N = 2^{cR^d}\leq \bar{N} \lesssim R^{(2s+\frac{5}{2}d-t)dR^{d-t}}M^{2R^{d-t}}R^{2(s+\frac{3d}{2}-t)R^{d-t}}.
 \end{equation}
 Taking the logarithm and keeping only the dependence on $R$, we get
 \begin{equation}
  R^d \lesssim (\log(R) + 1)R^{d-t},
 \end{equation}
 which yields a contradiction by taking $R\rightarrow \infty$, since $t > 0$.

\end{proof}

Finally, we consider lower bounds for ReLU$^k$ notworks. In the following theorem, we show that the maximal rate of $k-m+1$ obtained in Theorem \ref{piecewise-poly-approx-theorem} cannot be improved upon when approximating from $\Sigma^k_{n}$ regardless of the level of smoothness $s$. The argument is relatively straightforward and simply reduces to the one-dimensional case.

\begin{theorem}\label{relu-lower-bound}
 Let $\Omega = [0,1]^d$, $s\geq 0$, $k\in \mathbb{Z}_{\geq 0}$, and $0\leq m\leq s$. 
 Then there is a function $f\in \mathcal{B}^s(\Omega)$, such that
 \begin{equation}
  \inf_{f_n\in \Sigma^k_{n}} \|f - f_n\|_{H^m(\Omega)} \gtrsim n^{m-(k+1)}.
 \end{equation}
\end{theorem}
\begin{proof}
Consider the function
 \begin{equation}
  f(x) = e^{2\pi \iu x_1}\in \mathcal{B}^s(\Omega).
 \end{equation}
 This is a function of only one coordinate, which allows us to extend the lower bound in one dimension obtained in \cite{lin2014lower} to higher dimensions. 
 Specifically, we make the following simple observation,
 \begin{equation}
  \|f - f_n\|^2_{H^m(\Omega)} \geq \int_{\mathbb{R}^{d-1}} \|e^{2\pi \iu x_1} - f_n(x_1,x_{>1})\|^2_{H^m([0,1],dx_1)} dx_{>1}.
 \end{equation}
 This holds since derivatives with respect to variables other than $x_1$ will only increase the norm.
 
 Now, for each possible value of $x_{>1}$, $f_n(\cdot,x_{>1})$ is a one-dimensional piecewise polynomial function with at most $n$ breakpoints. Since $e^{2\pi \iu x_1}$ is not a piecewise polynomial function, the results in \cite{lin2014lower} imply that
 \begin{equation}
  \|e^{2\pi \iu x_1} - f_n(x_1,x_{>1})\|^2_{H^m([0,1],dx_1)} \gtrsim n^{-2(k-m+1)}.
 \end{equation}
 So we get
 \begin{equation}
 \begin{split}
  \|f - f_n\|^2_{H^m(\Omega)} &\geq \int_{\mathbb{R}^{d-1}} \|e^{2\pi \iu x_1} - f_n(x_1,x_{>1})\|^2_{H^m([0,1],dx_1)} dx_{>1} \\
  & \gtrsim n^{-2(k-m+1)},
  \end{split}
 \end{equation}
 as desired.

\end{proof}

\section{Conclusion}
We have shown that the approximation rates of neural networks with a cosine activation function or powers of a rectified linear unit as an activation function can be significantly improved beyond $O(n^{-\frac{1}{2}})$ for sufficiently smooth functions. In relation to the finite element method, this shows that a highly addaptive grid can lead to a significantly improved approximation rate for low degree piecewise polynomial functions. 

Further work which remains is understanding how much the approximation rates can be further improved when utilizing deeper networks. In particular, we believe that our techniques can be combined with the methods in \cite{bresler2020sharp,daubechies2019nonlinear} to obtain better rates for approximation of higher dimensional functions by deep neural networks.

\section{Acknowlegements}
We would like to thank Professors Weinan E, Ronald DeVore, and Russel Caflisch for helpful discussions, and Professor Yeonjong Shin for helpful comments on a draft of the manuscript. This work was supported by the Verne M. Willaman Fund, and the National Science Foundation (Grant No. DMS-1819157).

\bibliographystyle{spmpsci}
\bibliography{refs}

\end{document}